\documentclass[a4paper,10pt,leqno]{amsart}
\title[A solvable counterexample]{A solvable counterexample to the Hambleton-Taylor-Williams Conjecture}

\author{Julia Semikina}
        \address{Universit\"at Bonn\\
        Mathematisches Institut\\
                Endenicher Allee 60\\
                53115 Bonn, Germany}
         \email{juliasem@math.uni-bonn.de}
         
         \date{}


\theoremstyle{plain}         
\usepackage[all,cmtip]{xy}
\usepackage{amsmath,amssymb,amsthm}
\usepackage{mathtools}
\usepackage[T1]{fontenc}
\usepackage[english]{babel}
\usepackage{amscd}
\usepackage{dsfont}
\usepackage{hyperref}

\newtheorem{innercustomthm}{Theorem}
\newenvironment{customthm}[1]
  {\renewcommand\theinnercustomthm{#1}\innercustomthm}
  {\endinnercustomthm}

\newtheorem{theorem}{Theorem}
\newtheorem*{theorem*}{Theorem}
\newtheorem{definition}[theorem]{Definition}

\newtheorem{proposition}[theorem]{Proposition}
\newtheorem*{conjecture*}{Conjecture}
\newtheorem{lemma}[theorem]{Lemma}

\newtheorem{corollary}[theorem]{Corollary}

\newcommand{\End}{\mathrm{End}}
\newcommand{\Gal}{\mathrm{Gal}}
\newcommand{\Aut}{\mathrm{Aut}}
\newcommand{\rank}{\mathrm{rank}}
\newcommand{\SL}{\mathrm{SL}}
\newcommand{\ord}{\mathrm{ord}}

\newcommand{\overbar}[1]{\mkern 3mu\overline{\mkern-3mu#1\mkern-4mu}\mkern 4mu}

\begin{document}

\typeout{----------------------------  localization.tex  ----------------------------}

%%%%%%%%%%%%%%%%%%%%%%%%%%%%%%%%%%%%%%%%%%%%%%%%%%%%%%%%%%%%%%%%%%%%%%%%%%%%%%%%%
%%%%%%%%%%%%%%%%%%%%%%%%%%%%%%%%%%% Abstract  %%%%%%%%%%%%%%%%%%%%%%%%%%%%%%%%%%%%%%%
%%%%%%%%%%%%%%%%%%%%%%%%%%%%%%%%%%%%%%%%%%%%%%%%%%%%%%%%%%%%%%%%%%%%%%%%%%%%%%%%%

\typeout{------------------------------------ Abstract ----------------------------------------}

\begin{abstract}
In \cite{HTW} I. Hambleton, L. Taylor and B. Williams conjectured a general formula in spirit of H. Lenstra \cite{Le} for the decomposition of $G_n(RG)$ for any finite group $G$ and noetherian ring $R.$ The conjectured decomposition was shown to hold for some large classes of finite groups. In \cite{WeY} D. Webb and D. Yao discovered that the conjecture failed for the symmetric group $S_5$, but remarked that it still might be reasonable to expect the HTW-decomposition for solvable groups. In this paper we show that the solvable group $\SL(2, \mathds{F}_3)$ is also a counterexample to the conjectured HTW-decomposition. Furthermore, we prove that for any finite group $G$ the rank of $G_1(\mathds{Z}G)$ does not exceed the rank of the expression in the HTW-decomposition.
\end{abstract}

\maketitle

%%%%%%%%%%%%%%%%%%%%%%%%%%%%%%%%%%%%%%%%%%%%%%%%%%%%%%%%%%%%%%%%%%%%%%%%%%%%%%%%%
%%%%%%%%%%%%%%%%%%%%%%%%%%%%%%%%%% Introduction %%%%%%%%%%%%%%%%%%%%%%%%%%%%%%%%%%%%%
%%%%%%%%%%%%%%%%%%%%%%%%%%%%%%%%%%%%%%%%%%%%%%%%%%%%%%%%%%%%%%%%%%%%%%%%%%%%%%%%%

%\tableofcontents

 \typeout{-------------------------------   Section 0: Introduction --------------------------------}

\setcounter{section}{0}
\section*{Introduction}

In \cite{Le} Lenstra obtained a beautiful explicit formula for the decomposition of Grothendieck group $G_0(RG)$ of the group ring $RG$ for abelian group $G$ and noetherian ring $R.$ Namely, if $C(G)$ denotes the set of all cyclic quotients of $G$ (isomorphic quotients coming from different subgroups of $G$ are considered to be different), then 
$$G_0(RG) \cong \bigoplus_{C \in C(G)}   G_0 \Big( R \otimes _{\mathds{Z}} \mathds{Z}\big[ \xi_{|C|}, \frac{1}{|C|} \big] \Big),$$ where $|C|$ is the order of a cyclic group $C$, and $\xi_{|C|}$ is a primitive $|C|$-th root of unity. In \cite{We1}, \cite{We2} Webb proved the same decomposition formula for $G_n(RG), ~n>0$ for abelian group $G$, and obtained decomposition formulas for $G_n(\mathds{Z}G)$ for certain nonabelian groups $G$, in particular for dihedral and generalized quaternion groups \cite{We3}. In \cite{HTW} Hambleton, Taylor and Williams conjectured the general formula in spirit of Lenstra for the decomposition of $G_n(RG)$ for any finite group $G$ and noetherian ring $R.$ Their conjecture is consistent with all the previous results of Lenstra and Webb. Moreover it was shown to hold for finite nilpotent groups \cite{HTW}, \cite{We4} and groups of square-free order \cite{We5}. The conjecture was also proved for $G_0(\mathds{Z}G)$, where $G$ is a group of odd order having cyclic Sylow subgroups \cite{LaWe}. 

In order to state the conjecture let $G$ be a finite group, and let $\rho \colon G \to \Aut(V_{\rho})$ be a rational irreducible representation of $G.$ Then there is an associated division algebra $D_{\rho}=\End_{\mathds{Q}G}(V_{\rho})$ and we have the following Wedderburn decomposition of the rational group algebra (\cite{Se}, p. 92) 
$$\mathds{Q}G\cong \prod_{\rho} M_{n_{\rho}}(D_{\rho}^{op}),$$
where $\rho$ ranges over the set $X(G)$ of isomorphism classes of all rational irreducible representations of $G.$ 

For a such representation $\rho \colon G \to \Aut(V_{\rho}),$ let $k_{\rho}$ be the order of the kernel of the representation $\rho$ and let $d_{\rho}$ be the dimension of any of the irreducible complex constituents of $\mathds{C} \otimes_{\mathds{Q}} V_{\rho}.$ Let $\omega_{\rho}=\frac{|G|}{k_{\rho}d_{\rho}}.$ We remark that $\omega_{\rho}$ is an integer. 
 Indeed, the kernel of the irreducible rational representation $\rho$ coincides with the kernel of any of the irreducible complex constituents of $\mathds{C} \otimes_{\mathds{Q}} V_{\rho}.$ Hence such a constituent is a complex irreducible representation of the quotient group $G/ \ker{\rho}.$ Therefore, the integrality of $\omega_{\rho}$ follows from the fact that the dimension of a complex irreducible representation of a group divides the order of a group (\cite[p. 52]{Se}).

Let $\Lambda_{\rho}$ be a maximal $\mathds{Z}[1/ \omega_{\rho}]$-order in $D_{\rho}.$ Hambleton, Taylor and Williams conjectured the following decomposition formula, which we call HTW-decomposition.

\begin{conjecture*}[Hambleton-Taylor-Williams]
Let $G$ be a finite group and $R$ a noetherian ring. Then $$G_n(RG) \cong \bigoplus_{\rho \in X(G)} G_n(R \otimes \Lambda_{\rho}), ~~\forall n \geq 0.$$
\end{conjecture*}

It turned out that in general this conjecture does not hold. Webb and Yao \cite{WeY} showed that the formula failed for the symmetric group $S_5$, but remarked that it still might be reasonable to expect that the HTW-decomposition holds for solvable groups. In this paper we provide a solvable counterexample to the Hambleton-Taylor-Williams Conjecture. 
\begin{customthm}{A} \label{counter}
The group $\SL(2, \mathds{F}_3)$ does not satisfy the HTW-decomposition.
\end{customthm}
To prove this we use the same source of contradiction as in \cite{WeY}, namely, the rank of $G_1(\mathds{Z}G)$. For a finite group $G$ we consider the following two numbers 
\begin{itemize}
\item $R(G)$=the actual rank of $G_1(\mathds{Z}G)$ computed by Keating;
\item  $P(G)$=the rank of $G_1(\mathds{Z}G)$ predicted by the HTW-decomposition.
\end{itemize}
We start with an abstract description of $R(G)$ and $P(G)$ in Section 1.  In Section 2 we apply this description to compute the difference $P(G)-R(G)$ for the solvable group $G=\SL(2, \mathds{F}_3).$ The difference turns out to be non-zero and therefore we conclude that the group $\SL(2, \mathds{F}_3)$ is a counterexample to the Hambleton-Taylor-Williams Conjecture. 

In Section 3 we prove the general inequality estimating the number of modular irreducible representations of a finite group $G$ in terms of rational irreducible representations of $G.$ Let $E_{\rho}$ be the center of $D_{\rho}$ and let $\mathcal{O}_{\rho}$ be the ring of algebraic integers in $E_{\rho}.$
\begin{customthm}{B}\label{Inequality} 
Let $G$ be any finite group and let $p$ be a prime integer that divides the order of $G$. Then 
\[\# \{ \text{irreducible $\mathds{F}_{p}$-representations of $G$} \} \geq \sum_{\rho \in I_p} t_{\rho},\]
where $I_p$ is the set of all rational irreducible representations $\rho$ of $G$ for which the corresponding number $\omega_{\rho}$ is not divisible by $p$, and $t_{\rho}$ is the number of different prime ideals in $\mathcal{O}_{\rho}$ that divide the principal ideal $(p).$
\end{customthm}
As a corollary of Theorem~\ref{Inequality} we obtain that $P(G) \geq R(G)$ for any finite group $G.$ The proof of the inequality gives an explanation of the failure of the HTW-decomposition for $G_1(\mathds{Z}G).$

\subsection*{Acknowledgments.}

This work is a part of the author's ongoing dissertation project at the University of Bonn, it is supported by Bonn International Graduate School. The author thanks her supervisor Wolfgang L\"uck for his helpful suggestions and advices.

\section{The computation of ranks of $G_1$-groups}
\subsection{The description of $R(G)$}
Let $G$ be a finite group. An abelian group $G_1(\mathds{Z}G)$ is completely determined by Keating \cite{Ke}. To present Keating's formula for the rank of $G_1(\mathds{Z}G)$ we use the same notations as in the introduction. Namely, $\rho \colon G \to \Aut(V_{\rho})$ denotes a rational irreducible representation of $G,$ $D_{\rho}=\End_{\mathds{Q}G}(V_{\rho})$ is the corresponding division algebra, and rational group algebra decomposes as  
$\mathds{Q}G\cong \prod_{\rho \in X(G)} M_{n_{\rho}}(D^{op}_{\rho}).$
 
 \begin{theorem}[Keating's rank formula] \label{Ke}
Let $\Gamma_{\rho}$ be the maximal $\mathds{Z}$-order in the center of $D_{\rho}.$ Let $r_{\rho}$ be the rank of group of units in $\Gamma_{\rho},$ and let $v_{\rho}$ be the number of primes of $\Gamma_{\rho}$ that divide $|G|.$ Let $\varepsilon$ be the number of isomorphism classes of simple $\mathds{Z}G$-modules annihilated by $|G|.$ Then 
$$\rank\, G_1(\mathds{Z}G)= \sum _{\rho \in X(G)} (r_{\rho}+ v_{\rho}) -\varepsilon.$$

 \end{theorem}

Denote by $E_{\rho}$ the center of  $D_{\rho}.$ Obviously it is an algebraic number field and moreover $E_{\rho} / \mathds{Q}$ is an abelian extension (see \cite[Lemma 2.8]{BaRo}). Denote by $\mathcal{O}_{\rho}$ the ring of algebraic integers in $E_{\rho}$ (it is the maximal $\mathds{Z}$-order in $E_{\rho}$). 
Since $E_{\rho}$ is an abelian number field it is either totally real or totally imaginary. Let $m_{\rho}= [ E_{\rho} : \mathds{Q}].$
Then by Dirichlet's unit theorem the rank $r_{\rho}$ of the group of units in $\mathcal{O}_{\rho}$ is equal to $\#\{$real embeddings of $E_{\rho} \} +\# \{$conjugate pairs of complex embeddings of $E_{\rho} \}-1.$ 

Therefore
\begin{equation*}
                    r_{\rho}=\begin{cases}
                        m_{\rho}-1, \text{~if $E_{\rho}$ is totally real}, \\
                        m_{\rho} /2 -1, \text{~if $E_{\rho}$ is totally complex.}
                    \end{cases}
\end{equation*}

It is clear that

$$\varepsilon= \sum_{p | n} \# \{ \text{isomorphism classes of simple $\mathds{F}_p G$-modules}\},$$ 
where the sum ranges over all prime numbers that divide the order $n$ of the group $G$. To compute the number of irreducible $\mathds{F}_p$-representations of $G$ explicitly we first establish some notation. 

\begin{definition}
An element of a group $G$ is called $p$-regular if its order is not divisible by the prime number $p$.
\end{definition}

Denote by $d$ the L.C.M. of orders of all $p$-regular elements in $G$. 
%Let $\xi_d$ be a primitive $d$-th root of unity over $\mathds{F}_p$. The Galois group $\Gal(\mathds{F}_p(\xi_d)/\mathds{F}_p)$ is a cyclic group generated by an $\mathds{F}_p$-automorphism $\sigma$ sending $\xi_d$ to $\xi^{p}_d.$ 
Let $k$ be the smallest positive integer such that $p^k=1$ (mod $d$). 
%then $$\Gal(\mathds{F}_p(\xi_d)/\mathds{F}_p)=\{ \sigma^{i} ~|~i=1,2,\ldots,k \}.$$ 
Denote by $T$ the multiplicative group of exponents $\{{p^i} ~|~ i=1,2,\ldots,k \}$ modulo $d$.

\begin{definition}
Two $p$-regular elements $g_1, g_2  \in G$ are called $\mathds{F}_p$-conjugate if $g^{t}_1=h g_2 h^{-1}$ for some $t \in T$ and  $h \in G.$
\end{definition}

 \begin{theorem}[Berman \cite{Be}]
 The number of irreducible representations of $G$ over $\mathds{F}_p$ equals the number of $p$-regular $\mathds{F}_p$-conjugacy classes.
 
 \end{theorem}
 
 Hence $$\varepsilon= \sum_{p | n} \# \{ \text{$p$-regular $\mathds{F}_p$-conjugacy classes} \}.$$ 
 
 Now we have very computable description of the summands appearing in Theorem~\ref{Ke} and we denote by $R(G)$ the rank of $G_1(\mathds{Z}G).$ More precisely we have
 
 $$R(G) = \sum_{\rho \in X(G)} (r_{\rho} + v_{\rho}) -\varepsilon.$$
 
\subsection{The description of $P(G)$}

In this subsection we are going to compute the rank of $G_1(\mathds{Z}G)$ as predicted by the Hambleton-Taylor-Williams Conjecture. If HTW-decomposition holds for a group $G$ and any noetherian coefficient ring $R$, then taking $R=\mathds{Z}$ gives 
 $$G_1(\mathds{Z}G) \cong \bigoplus_{\rho \in X(G)} G_1(\Lambda_{\rho})$$
 and
 \[ \rank \, G_1(\mathds{Z}G) = \sum_{\rho \in X(G)} \rank\, G_1(\Lambda_{\rho}).
\]

Denote by $P(G)$ the rank of $G_1(\mathds{Z}G)$ that is predicted by the HTW-decomposition, i.e., $P(G)\coloneqq \sum_{\rho \in X(G)} \rank\, G_1(\Lambda_{\rho})$.  We are going to give an explicit description of $P(G)$ and compare it with the real rank $R(G)$ computed in the previous subsection. Obviously, if $P(G) \neq R(G)$, then the group $G$ does not satisfy the HTW-decomposition.

For a rational irreducible representation $\rho$ we determine the rank of $G_1(\Lambda_{\rho})$ by applying the following theorem (\cite{L}, \cite{SwEv}).

 \begin{theorem}[Lam] \label{Lam}
Let $R$ be a Dedekind ring with quotient field $K$ and let $A$ be a separable semisimple $K$-algebra such that

\begin{enumerate}

\item[(i)] $R/\mathfrak{p}$ is finite for every non-zero prime ideal $\mathfrak{p} \subset R$, and

\item[(ii)] if $L$ is a finite separable field extension of $K$ and $S$ is the integral closure of $R$ in $L,$ then the class group $Cl(S)$ is a torsion group. 

\end{enumerate}
Let $Z(A)$ be the center of $A$, $\tilde{R}$ be the integral closure of $R$ in $Z(A),$ and $\Lambda$ be an $R$-order in $A.$
Then $$G_1(\Lambda) \underset{\text{mod torsion}} {\cong}  \tilde{R}^{\times}.$$

 \end{theorem}

Let $R=\mathds{Z}[1/ \omega_{\rho}],~K=\mathds{Q}$ and $A=D_{\rho}.$ The prime ideals in the localization $\mathds{Z}[1/ \omega_{\rho}]$ are of the form $q \mathds{Z}[1/ \omega_{\rho}],$ where $q$ is a prime integer that does not divide $\omega_{\rho}$. Hence the first condition of Theorem~\ref{Lam} is satisfied, because $$\mathds{Z}[1/ \omega_{\rho}] / q \mathds{Z}[1/ \omega_{\rho}] \cong (\mathds{Z} / q \mathds{Z})[1/ \omega_{\rho}] \cong \mathds{Z} / q \mathds{Z}.$$

If $L$ is an algebraic number field and $\mathcal{O}_{L}$ is a ring of algebraic integers in $L$, then the integral closure of $\mathds{Z}[1/ \omega_{\rho}]$ in $L$ is $\mathcal{O}_{L}[1/ \omega_{\rho}].$ To check that the second condition of Theorem~\ref{Lam} holds it is enough to show that $Cl(\mathcal{O}_{L}[1/ \omega_{\rho}])$ is finite. Recall that for a Dedekind ring $\mathcal{R}$ the ideal class group of $\mathcal{R}$ is given by the quotient $Cl(\mathcal{R})=\mathcal{I}_{\mathcal{R}} / \mathcal{P}_{\mathcal{R}},$ where $\mathcal{I}_{\mathcal{R}}$ is the group of fractional ideals of $\mathcal{R}$ and $\mathcal{P}_{\mathcal{R}}$ is the subgroup of principal fractional ideals of $\mathcal{R}.$ Let $\phi \colon \mathcal{I}_{\mathcal{O}_{L}} \to \mathcal{I}_{\mathcal{O}_{L}[1/\omega_{\rho}]}$ be a map sending a fractional ideal $I$ of $\mathcal{O}_{L}$ to the fractional ideal  $\mathcal{O}_{L}[1/ \omega_{\rho}] \otimes_{\mathcal{O}_{L}} I $ of $\mathcal{O}_{L}[1/ \omega_{\rho}].$ This map is obviously a group homomorphism. Moreover $\phi$ takes principal fractional ideals of $\mathcal{O}_{L}$ to principal fractional ideals of $\mathcal{O}_{L}[1/\omega_{\rho}]$ and is surjective,  since $\mathcal{O}_{L}$ and $\mathcal{O}_{L}[1/\omega_{\rho}]$ have the same field of fractions. Therefore $\phi$ induces a surjective group homomorphism 

$$\bar{ \phi} \colon Cl (\mathcal{O}_{L}) \twoheadrightarrow Cl \big( \mathcal{O}_{L}  \Big[ \frac{1}{\omega_{\rho}} \Big] \big).$$

  Now finiteness of the group $Cl (\mathcal{O}_{L}[1/\omega_{\rho}])$ follows from the surjectivity of $\bar{\phi}$ and the classical result that the ideal class group of the ring of integers in an algebraic number field is finite.
  
 The integral closure of $\mathds{Z}[1/\omega_{\rho}]$ in $Z(D_{\rho})=E_{\rho}$ is $\mathcal{O}_{\rho}[1/\omega_{\rho}].$ Therefore Theorem~\ref{Lam} implies $$G_1(\Lambda_{\rho}) \underset{\text{mod torsion}}{\cong} \big( \mathcal{O}_{\rho} \Big[ \frac{1}{\omega_{\rho}} \Big] \big) ^{\times}.$$
 
To determine the rank of units in $\mathcal{O}_{\rho}[1/\omega_{\rho}]$ we use the $S$-unit theorem (see, for example, \cite[p. 88]{Ne}). Let us briefly recall its statement. Let $L$ be an algebraic number field and let $S$ be a finite set of prime ideals in $\mathcal{O}_{L}.$ Define $$\mathcal{O}_{L}(S) \coloneqq \{ x \in L ~| ~ord_{\mathfrak{p}}(x) \geq 0, \text{ for all prime ideals } \mathfrak{p} \notin S \}.$$
 Define the group $U(S)$ of $S$-units to be 
 $$U(S) \coloneqq \big( \mathcal{O}_{L}(S)\big) ^{\times}=\{ x \in L ~| ~ord_{\mathfrak{p}}(x)=0, \text{ for all } \mathfrak{p} \notin S \}.$$

\begin{theorem}[S-unit theorem] \label{S}
The group of S-units is finitely generated and $$\rank \, U(S) =\rank \, \mathcal{O}^{\times}_{L}   + \#S.$$
\end{theorem}
 
Let $S$ be the set of all prime ideals in $\mathcal{O}_{\rho}$ that divide $\omega_{\rho}.$ Then  $\mathcal{O}_{E_{\rho}}(S)=\mathcal{O}_{\rho}[1/\omega_{\rho}]$ and $U(S)=\big( \mathcal{O}_{\rho} [ 1/\omega_{\rho} ] \big) ^{\times}$. Hence by Theorem~\ref{S}
$$\rank\,\big( \mathcal{O}_{\rho} [ 1/\omega_{\rho} ] \big) ^{\times}=\rank\, \mathcal{O}^{\times}_{\rho}+\#\{\text{prime ideals in } \mathcal{O}_{\rho} \text{ that divide } \omega_{\rho}\}$$
and finally
$$\rank\,G_1(\Lambda_{\rho})=r_{\rho}+\#\{\text{prime ideals in } \mathcal{O}_{\rho} \text{ that divide } \omega_{\rho}\}.$$
Let us denote by $w_\rho$ the number of prime ideals in $\mathcal{O}_{\rho}$ that divide $\omega_{\rho}.$ Then
\[
P(G) = \sum_{\rho \in X(G)} \rank\,G_1(\Lambda_{\rho})=\sum_{\rho \in X(G)} \big( r_{\rho}+  w_{\rho} \big).
\]

\section{A solvable counterexample}

In this section we show that $G=\SL(2, \mathds{F}_{3}),$ the group of $2 \times 2$ matrices with determinant 1 over the finite field $\mathds{F}_{3},$ is a counterexample to the Hambleton-Taylor-Williams Conjecture by comparing $R(G)$ and $P(G)$ that were described in the previous section. Note that the order of the group $G$ is 24.

For computation of $R(G)$ and $P(G)$ we need to know the table of complex irreducible characters of $G.$ If $\rho$ is an irreducible rational representation of $G$, then $E_{\rho}$  the center of the corresponding division algebra $D_{\rho}$ appearing in the Wedderburn decomposition is isomorphic to the field of character values $\mathds{Q}(\chi^{\mathds{C}}_{\rho}) \coloneqq \mathds{Q}(\chi^{\mathds{C}}_{\rho}(g)~|~g \in G),$ where $\chi^{\mathds{C}}_{\rho}$ is a character of any irreducible complex constituent of the complexification of the representation $\rho$ (see \cite[Theorem 3.3.1]{Je}).

Denote by $\xi$ a primitive cube root of unity. Complex irreducible characters of $\SL(2, \mathds{F}_3)$ are well-known (see \cite[p. 132]{Bo}) and are presented in Table 1.

\begin{table}[h!]
\caption {Character table of  $\SL(2, \mathds{F}_{3})$}
  \centering
  \begin{tabular}{l| c c c c c c c|| c }
repr.:&  $\big( \begin{smallmatrix} 1 & 0  \\ 0 & 1 \end{smallmatrix} \big)$ & $\big( \begin{smallmatrix} -1 & 0  \\ 0 & -1 \end{smallmatrix} \big)$ & $\big( \begin{smallmatrix} 0 & -1  \\ 1 & 0 \end{smallmatrix} \big)$ & $\big( \begin{smallmatrix} -1 & 1  \\ 0 & -1 \end{smallmatrix} \big)$& $\big( \begin{smallmatrix} 1 & -1  \\ 0 & 1 \end{smallmatrix} \big)$& $\big( \begin{smallmatrix} -1 & -1  \\ 0 & -1 \end{smallmatrix} \big)$ & $\big( \begin{smallmatrix} 1 & 1  \\ 0 & 1 \end{smallmatrix} \big)$ & $\mathds{Q}(\chi)$ \\
    size:&1 & 1 & 6 & 4 & 4 & 4 & 4&\\
    
    order:&1 & 2 & 4 & 6 & 3 & 6 & 3&\\
    \hline
    $\chi_1$ & 1 & 1 & 1& 1& 1& 1& 1 &$\mathds{Q}$\\
    $\chi_2$ & 3 & 3 & -1& 0& 0& 0& 0 &$\mathds{Q}$\\
    $\chi_3$ & 2 & -2 & 0& 1& -1& 1& -1 &$\mathds{Q}$\\
    $\chi_4$ & 1 & 1 & 1& $\xi$& $\xi$& $\xi^2$& $\xi^2$ &$\mathds{Q}(\xi)$\\
    $\chi_5$ & 1 & 1 & 1& $\xi^2$& $\xi^2$& $\xi$& $\xi$ &$\mathds{Q}(\xi)$\\
    $\chi_6$ & 2 & -2 & 0& $\xi$& -$\xi$& $\xi^2$& -$\xi^2$ &$\mathds{Q}(\xi)$\\
    $\chi_7$ & 2 & -2 & 0& $\xi^2$& -$\xi^2$& $\xi$& -$\xi$ &$\mathds{Q}(\xi)$\\
  \end{tabular}
\end{table}

The field $\mathds{Q}(\xi)$ is the splitting field for the group $G.$ The Galois group $\Gal(\mathds{Q}(\xi)/\mathds{Q})$ acts on the set of complex representations of $G$, in particular it permutes irreducible complex representations. Let $\chi$ be a character of irreducible rational representation of $G$. It is well known (see \cite[(74.5)]{CuRe2}) that $\chi$ can be realized as a sum of all distinct Galois conjugates of some complex irreducible character $\varphi$ taken with multiplicity $m(\varphi),$ which is the Schur index of $\varphi$, i.e. $\chi=m(\varphi) \sum_{\sigma} \varphi^{\sigma}.$ And vice versa any complex irreducible representation gives rise to a unique rational irreducible representation in the way described above.

The Galois group $\Gal(\mathds{Q}(\xi)/\mathds{Q})$ permutes characters $\chi_4$ with $\chi_5$, and $\chi_6$ with $\chi_7.$ The character $\chi_3$ is fixed by $\Gal(\mathds{Q}(\xi)/\mathds{Q})$, but the corresponding representation is not defined over $\mathds{Q}$ and has Schur index 2. We have 2 absolutely irreducible rational representations $\rho_1, \rho_2$ with characters $\chi_1, \chi_2,$ respectively; and 3 irreducible rational representations $\rho_3, \rho_4, \rho_5,$ s.t. $\chi_3$ is a character of one of the irreducible complex constituents of $\rho^{\mathds{C}}_3$, $\chi_4$ is a character of one of the irreducible complex constituents of $\rho^{\mathds{C}}_4$, and $\chi_6$ is a character of one of the irreducible complex constituents of $\rho_5^{\mathds{C}}$ (here $\rho^{\mathds{C}}$ denotes the complexification $\mathds{C} \otimes_{\mathds{Q}} \rho$ of a representation $\rho$).
 
The kernel of a representation $\rho$ coincides with the kernel of the induced character $\ker \chi_{\rho}=\{ g \in G~|~\chi_{\rho}(g)=\chi_{\rho}(e) \}$. Moreover, by the way how the rational irreducible representation is expressible in terms of complex irreducible representations the kernel of $\rho$ coincides with the kernel of $\chi^{\mathds{C}}_{\rho}.$ Therefore we have the following list of centers $E_{\rho}$'s and values of $\omega_{\rho}$'s given in Table 2.

\begin{table}[h!]  
\caption {}
  \centering
  \begin{tabular}{l|c c c c c }
 & $\rho_1$ & $\rho_2$ & $\rho_3$ & $\rho_4$ & $\rho_5$\\

      \hline
   $E_{\rho}$ & $\mathds{Q}$ & $\mathds{Q}$ & $\mathds{Q}$ & $\mathds{Q}(\xi)$ & $\mathds{Q}(\xi)$\\

  $\mathcal{O}_{\rho}$ & $\mathds{Z}$ & $\mathds{Z}$ & $\mathds{Z}$ & $\mathds{Z}[\xi]$ & $\mathds{Z}[\xi]$\\

  $\omega_{\rho}=\frac{24}{\ker \chi^{\mathds{C}}_{\rho} ~ \dim \chi^{\mathds{C}}_{\rho}}$ &1 &4 &12 &3 &12\\

  $v_{\rho}=\#\{$prime ideals in $\mathcal{O}_{\rho}$ that divide 24$\}$ &2 &2 &2 &2 &2\\

  $w_{\rho}=\# \{$prime ideals in  $\mathcal{O}_{\rho}$ that divide $\omega_{\rho}\}$ &0 &1 &2 &1 &2\\
  \end{tabular}
\end{table}

For computation of numbers  $v_{\rho}$ and $w_{\rho}$ in Table 2 we used the following theorem (for a more general statement see \cite[p. 47]{Ne}) to determine the number of prime ideals in $\mathds{Z}[\xi]$ that divide 2 and 3.

\begin{theorem} \label{prime} Let $\alpha$ be an algebraic integer such that $\mathds{Z}[\alpha]$ is integrally closed, and let $f$ be the minimal polynomial of $\alpha.$ Let $p$ be a prime number and let $$f(x)=\prod_{i} f_i(x)^{e_i}$$
in $\mathds{F}_p[x].$ Then the prime ideals that lie above $p$ in $\mathds{Z}[\alpha]$ are precisely the ideals $(p, f_i(\alpha)).$ 
\end{theorem}
For $p=3$ the minimal polynomial of $\xi$ factors as $x^2+x+1=(x-1)^2$ in $\mathds{F}_3[x]$ and hence by Theorem~\ref{prime} there is only one prime ideal above 3 in $\mathds{Z}[\xi].$ For $p=2$ the polynomial $x^2+x+1$ is irreducible in $\mathds{F}_2[x]$ and hence there is exactly one prime ideal above 2 in $\mathds{Z}[\xi].$

The last step is to determine $\varepsilon$ appearing in the description of $R(G)$. For this we need to compute the number of $p$-regular $\mathds{F}_p$-conjugacy classes for $p=2,3.$

For $p=3$ there are 3 conjugacy classes of 3-regular elements, namely the classes with representatives $\big( \begin{smallmatrix} 1 & 0  \\ 0 & 1 \end{smallmatrix} \big),\big( \begin{smallmatrix} -1 & 0  \\ 0 & -1 \end{smallmatrix} \big)$ and $\big( \begin{smallmatrix} 0 & -1  \\ 1 & 0 \end{smallmatrix} \big).$ Since first two elements belong to the center of $G$ and $\big( \begin{smallmatrix} 0 & -1  \\ 1 & 0 \end{smallmatrix} \big)^3=\big( \begin{smallmatrix} 0 & -1  \\ 1 & 0 \end{smallmatrix} \big),$ these elements belong to 3 different $\mathds{F}_3$-conjugacy classes.

For $p=2$ there are 3 conjugacy classes of 2-regular elements, namely the classes with representatives $\big( \begin{smallmatrix} 1 & 0  \\ 0 & 1 \end{smallmatrix} \big),\big( \begin{smallmatrix} 1 & 1  \\ 0 & 1 \end{smallmatrix} \big)$ and $\big( \begin{smallmatrix} 1 & -1  \\ 0 & 1 \end{smallmatrix} \big).$ Since $\big( \begin{smallmatrix} 1 & 1  \\ 0 & 1 \end{smallmatrix} \big)^2=\big( \begin{smallmatrix} 1 & -1  \\ 0 & 1 \end{smallmatrix} \big)$ the elements $\big( \begin{smallmatrix} 1 & 1  \\ 0 & 1 \end{smallmatrix} \big)$ and $\big( \begin{smallmatrix} 1 & -1  \\ 0 & 1 \end{smallmatrix} \big)$  are $\mathds{F}_2$-conjugated. Therefore there are two 2-regular $\mathds{F}_2$-conjugacy classes. Finally

$$P(G)-R(G)=\sum_{\rho}(r_{\rho} + w_{\rho}) -\sum_{\rho} (r_{\rho} + v_{\rho}) + \varepsilon=\sum_{\rho} w_{\rho}-\sum_{\rho} v_{\rho}+ \varepsilon=6-10+5=1.$$

This shows that the actual rank of $G_1(\mathds{Z}G)$ and the rank predicted by the HTW-conjecture do not coincide and $G=\SL(2, \mathds{F}_3)$ is a solvable counterexample to the conjectured formula. Thus Theorem~\ref{counter} is proved.

\section{Counting modular representations in terms of rational}

In this section we prove an inequality estimating the number of modular irreducible representations of a finite group $G$ in terms of rational irreducible representations of $G.$ The result that we obtain implies that $P(G) \geq R(G)$ for any finite group $G.$  The proof of the inequality provides an explanation why in general the HTW-Conjecture does not hold for $G_1(\mathds{Z}G)$. We are using the same notation as before. 

Let $p$ be a prime number that divides the order of $G$. Recall that for any rational irreducible representation $\rho$ of $G$ there is an associated algebraic number field $E_{\rho} \colon$ the field of character values $\mathds{Q}(\chi^{\mathds{C}}_{\rho}),$ where $\chi^{\mathds{C}}_{\rho}$ is a character of any of the irreducible complex constituents of $\mathds{C} \otimes_{\mathds{Q}} \rho$. Let $t_{\rho}$ be the number of different prime ideals in $\mathcal{O}_{\rho}$ that divide the principal ideal $(p).$ Then the following inequality holds.

\begin{customthm}{B}
Let $G$ be any finite group and let $p$ be a prime integer that divides the order of $G$. Then 
\[\# \{ \text{irreducible $\mathds{F}_{p}$-representations of $G$} \} \geq \sum_{\rho \in I_p} t_{\rho},\]
where $I_p$ is the set of all rational irreducible representations $\rho$ of $G$ for which the corresponding number $\omega_{\rho}$ is not divisible by $p$.
\end{customthm}

The main ingredient in the proof of Theorem~\ref{Inequality} is the following classical theorem due to Brauer and Nesbitt (see \cite{BrNe}, \cite{CuRe}).
\begin{theorem}[Brauer-Nesbitt]
Let $G$ be a finite group of order $|G|=p^a m$, where $(p,m)=1$. If a complex irreducible representation $\phi$ has dimension divisible by $p^a,$ then it remains irreducible after reduction mod p. Moreover, the character of $\phi$ vanishes on all elements of $G$, whose order is divisible by $p,$ and coincides with Brauer character of $\bar{\phi}$ (the reduction mod $p$ of $\phi$) on $p$-regular elements of $G$.
\end{theorem}
We start with some preparatory statements and definitions before giving the proof of Theorem~\ref{Inequality}. 
It is well known (see \cite[p. 284]{CuRe}) that any element $g \in G$ is expressible as $g=g_{p'} g_{p},$ where $g_{p'}$ and $g_{p}$ commute, $g_{p'}$ has order coprime with $p$, and $g_{p}$ has order a power of $p$. The elements $g_{p'}, g_{p}$ are uniquely determined and are called $p$-regular and $p$-singular components of $g$, respectively.

\begin{lemma} \label{quotient-order}
Let $H$ be a normal subgroup in $G$, and let $g=g_{p'} g_{p} \in G.$  If $g_{p} \not \in H,$ then $\ord_{G/H}(\bar{g})$ is divisible by $p$, where $\ord_{G/H}(\bar{g})$ is the order of $g$ considered as an element in the quotient group $G/H.$
\end{lemma}

\begin{proof}
For an element $x \in G$ we denote by $\bar{x}$ the image of $x$ in the quotient group $G/H$. Let $a=\ord_{G/H}(\bar{g})$. Since elements $g_{p'}$ and $g_{p}$ commute in $G$ we have $\bar{g}^a=\overbar{g_{p'}}^a \overbar{g_{p}}^a=e.$ Hence $(\overbar{g_{p'}}^{-1})^a= \overbar{g_{p}}^a.$ The order of the element $\overbar{g_{p}}^a$ in $G/H$ divides the order of $g^a_{p}$ in $G,$ and therefore $\ord_{G/H}(\overbar{g_{p}}^a)$ is a power of $p.$ The same way the order of $(\overbar{g_{p'}}^{-1})^a$ in $G/H$ is coprime with $p.$ This implies that $(\overbar{g_{p'}}^{-1})^a= \overbar{g_{p}}^a=e$ and hence $a$ is divisible by $\ord_{G/H}(\overbar{g_{p}}).$ At the same time $\overbar{g_{p}} \neq e$ and hence $\ord_{G/H}(\overbar{g_{p}})$ is divisible by $p.$ Therefore $a$ is divisible by $p.$  
\end{proof}

\begin{lemma}\label{character-criterion}
Let $\varphi_1, \varphi_2$ be complex irreducible characters of $G.$ Let $H_1=\ker(\varphi_1)$ and $H_2=\ker(\varphi_2)$. Suppose the following conditions are satisfied.
\begin{enumerate}
\item $\varphi_i(g)=0$ for all $g \in G,$ s.t. $\ord_{G/H_i}(\bar{g})$ is divisible by $p$, $i \in \{1,2\}$;
\item $\varphi_1(g)=\varphi_2(g)$ for all $g \in G,$ s.t. $\ord_G(g)$ is not divisible by $p$.
\end{enumerate}
Then $\varphi_1$ and $\varphi_2$ are equal.
\end{lemma}

\begin{proof}
Any element $g \in G$ is expressible in the form $g=g_{p'} g_{p},$ where $g_{p'}$ is $p$-regular and $g_{p}$ is $p$-singular. If $g_{p} \not \in H_1$, then by Lemma~\ref{quotient-order} we have that  $\ord_{G/H_1}(\bar{g})$ is divisible by $p$, and therefore by the first condition $\varphi_1(g)=0.$ Analogously, if $g_{p} \not \in H_2$, then $\varphi_2(g)=0.$ Hence we have the following equalities 

\begin{equation}
    \varphi_1(g)=\varphi_1(g_{p'} g_{p})=
    \begin{cases}
      0, & \text{if}\ g_{p} \not \in H_1\\
      \varphi_1(g_{p'})=\varphi_2(g_{p'}), & \text{if}\ g_{p} \in H_1
    \end{cases}
  \end{equation} 
  and
  \begin{equation}
    \varphi_2(g)=\varphi_2(g_{p'} g_{p})=
    \begin{cases}
      0, & \text{if}\ g_{p} \not \in H_2\\
      \varphi_2(g_{p'})=\varphi_1(g_{p'}), & \text{if}\ g_{p} \in H_2.
    \end{cases}
  \end{equation} 
  Let us consider the inner product of the characters $\varphi_1, \varphi_2$
  
  \begin{equation} \langle \varphi_1, \varphi_2 \rangle=\frac{1}{|G|} \sum_{g \in G} \varphi_1(g)\overline{\varphi_2(g)}=\frac{1}{|G|} \sum_{\substack{g=g_{p'}g_p \in G,\\           
   g_{p} \in {H_1 \cap H_2}}} \varphi_1(g_{p'}) \overline{\varphi_1(g_{p'})},
  \end{equation}
where the last equality holds since all terms with $g_{p} \not \in {H_1 \cap H_2}$ vanish. The value of the sum is a positive real number, because each term is a non-negative real number and taking $g$ to be the identity element $e$ gives a non-zero summand $|\varphi_1(e)|^2$. This implies that the inner product of irreducible characters $\varphi_1, \varphi_2$ is non-zero, and  therefore by the orthogonality theorem $\varphi_1=\varphi_2.$  
\end{proof}

Let $K$ be an algebraic number field that is a splitting field for $G$ and is a Galois extension of $\mathds{Q},$ e.g. by Brauer's Realization Theorem we may take $K=\mathds{Q}(\xi_{|G|}),$ where $\xi_{|G|}$ is a primitive $|G|$-th root of unity. This means that any complex representation of $G$ is realized over $K.$ Let $\mathcal{O}_K$ denote the ring of algebraic integers of $K$ and let $\mathfrak{p}$ be a prime ideal in $\mathcal{O}_K$ containing $p.$ Denote the field $\mathcal{O}_K / \mathfrak{p}$ by $\bar{K},$ it is clearly a finite Galois extension of $\mathds{F}_p.$ Recall that with any representation $\phi$ over $K$ we can associate a representation $\bar{\phi}$ over $\bar{K}$, whose composition factors are uniquely determined (see e.g. \cite[\S 82]{CuRe}). We refer to this process as reduction mod $p$. Let $d \colon G_0(KG) \to G_0(\bar{K}G)$ be the map induced by the reduction mod $p.$

\begin{definition}
Let $\phi$ be an irreducible complex representation of $G.$ We call $\phi$ $p$-special if the corresponding number $\omega_{\phi}=\frac{|G/ \ker \phi|}{\dim \phi}$ is not divisible by $p.$
\end{definition}
The next lemma shows that reduction mod $p$ is injective on the set of $p$-special representations. Whenever we mention the set of representations, we mean the set of isomorphism classes of representations.

\begin{proposition}\label{Injectivity}
Let $G$ be a finite group, $p$ be a prime number that divides $|G|,$ and let $K, \bar{K}$ be defined as above. Then the following holds.
\begin{enumerate}
\item[(i)] The reduction mod $p$ of a $p$-special representation of $G$ is irreducible $\bar{K}$-representation of $G$.
\item[(ii)] The restriction of the map $d \colon G_0(KG) \to G_0(\bar{K}G)$ to the set of $p$-special representations of $G$ is injective.
\end{enumerate}
\end{proposition}

\begin{proof} (i) Let $\phi$ be a $p$-special representation of $G$. Denote by $H$ the kernel of $\phi.$ Then $\phi$ induces a representation of $G/H,$ which we denote by $\phi_{G/H}.$ Since $\phi$ is irreducible, the representation $\phi_{G/H}$ is irreducible as well. Now the Brauer-Nesbitt Theorem can be applied to $\phi_{G/H},$ because $\frac{|G/H|} {\dim \phi_{G/H}}$ is not divisible by $p$ by the definition of a $p$-special representation. Therefore $\bar{\phi}_{G/H}$ is an irreducible $\bar{K}$-representation of $G/H$, which implies that $\bar{\phi}$ is an irreducible  $\bar{K}$-representation of $G.$ 

(ii) The Brauer-Nesbitt Theorem guarantees that the character of $\phi_{G/H}$ vanishes on all elements of $G/H$ having order divisible by $p,$ and coincides with Brauer character of $\bar{\phi}_{G/H}$ on $p$-regular elements of $G/H.$ Note that $p$-regular elements of $G$ remain $p$-regular when passing to the quotient group $G/H.$

Suppose that the reduction mod $p$ of two $p$-special representations $\phi_1$ and $\phi_2$ gives $\bar{K}$-equivalent representations $\bar{\phi}_1, \bar{\phi}_2.$ Then the Brauer characters of $\bar{\phi}_1$ and $\bar{\phi}_2$ are the same, which implies by Brauer-Nesbitt Theorem that ordinary characters of $\phi_1$ and $\phi_2$ coincide on $p$-regular elements of $G.$ This means that conditions of Lemma \ref{character-criterion} are satisfied for the characters of $\phi_1$ and $\phi_2,$ hence representations $\phi_1$ and $\phi_2$ are $K$-equivalent. Therefore the restriction of the map $d$ to the set of $p$-special representations of $G$ is injective. 
\end{proof}

To prove Theorem~\ref{Inequality} we will also use the following theorem, that describes the behavior of irreducible representations under extension of the base field (see \cite[(74.5)]{CuRe2}).

\begin{theorem}\label{orbit}
Let $G$ be a finite group, $k$ arbitrary field, and let $E$ be a splitting field for $G,$ s.t. $E$ is a finite Galois extension of $k.$ Then we have
\begin{itemize}
\item[(i)] For a simple $kG$-module $U$ there is an isomorphism of $EG$-modules 
\[ E \otimes_{k} U \cong  (V_1 \oplus \ldots \oplus V_t)^{\oplus m},
\]
where $\{V_1, \ldots, V_t\}$ is a set of non-isomorphic simple left $EG$-modules permuted transitively by the Galois group $\Gal(E/k),$ and $m=m_k(V_i)$ is the Schur index.
\item[(ii)] Let $\varphi$ be an absolutely irreducible character of $G$ afforded by some simple left $EG$-module $V.$ Then there is a simple $kG$-module $U$, unique up to isomorphism, s.t. $\varphi$ occurs in the character $\chi$ afforded by $U.$ Then $V$ occurs as a summand in the decomposition of $E \otimes_{k} U$, say $V=V_1$ and \[\chi=m \sum_{i=1}^{t} \varphi_i, ~~~k(\varphi_1) \cong \ldots \cong k(\varphi_t),\] where $\varphi_i$ is the character of $G$ afforded by $V_i,$ and $k(\varphi_i)$ is the field of character values.
\end{itemize}
\end{theorem}
Now we are ready to prove Theorem~\ref{Inequality}.
\begin{proof}
The assumption of the Theorem~\ref{orbit} is satisfied for $k=\mathds{Q}$ and $E=K,$ therefore, it gives one to one correspondence between rational irreducible representations of $G$ and orbits of the Galois group $\Gal(K / \mathds{Q})$ action on the set of irreducible $K$-representations of $G$. In particular, it gives one to one correspondence between the set $I_p$ and orbits of the $\Gal(K / \mathds{Q})$-action on the set of $p$-special representations of $G.$ 

Given a $p$-special representation $\phi$ of $G$ let $\varphi$ be its character. Denote by $E_{\phi}$ the field of character values $\mathds{Q}(\varphi).$ Then by \cite[Theorem 70.15]{CuRe} the size of the orbit of $\phi$ under the action of $\Gal(K/\mathds{Q})$ is given by
\[|Orb(\phi)|=|E_\phi : \mathds{Q}|=|\Gal(E_\phi / \mathds{Q})|.\]

Next we examine the orbit $Orb(\phi)$ after the reduction mod $p.$ Since the irreducible representation $\phi$  is $p$-special the Proposition~\ref{Injectivity} implies that all Galois conjugates of $\phi$ remain irreducible when reduced mod $p$. Denote by $\bar{\phi}$ the reduction mod~$p$ of $\phi,$ and by $\overline{Orb(\phi)}$ the set $Orb(\phi)$ after the reduction mod $p$. By Proposition~\ref{Injectivity} the set $\overline{Orb(\phi)}$ consists of pairwise non-isomorphic irreducible $\bar{K}$-representations. This implies that the set $\overline{Orb(\phi)}$ has the same number of elements as $Orb(\phi).$ 

It is well known that the decomposition group $\mathfrak{D}(\mathfrak{p}/p)=\{\sigma \in \Gal(K/\mathds{Q}) ~|~ \sigma(\mathfrak{p})=\mathfrak{p}\}$ naturally surjects onto the Galois group $\Gal( \bar{K} / \mathds{F}_p)$ with the kernel being the inertia group $\mathfrak{I}(\mathfrak{p}/p)$ (see e.g. \cite[\S 7]{Se2}). Therefore the Galois group $\Gal( \bar{K} / \mathds{F}_p )$ preserves the set $\overline{Orb(\phi)}.$

Let us see how many orbits does this action have. Let $\bar{\phi^{\sigma}} \in \overline{Orb(\phi)},~\sigma \in \Gal(K/\mathds{Q}).$  Again by \cite[Theorem 70.15]{CuRe} the size of the orbit of $\bar{\phi^{\sigma}}$ under the action of $\Gal( \bar{K} / \mathds{F}_p )$ is given by
\[|Orb(\bar{\phi^{\sigma}})|=|\mathds{F}_p(\bar{\varphi^{\sigma}}) : \mathds{F}_p|,\] where $\bar{\varphi^{\sigma}}$ is the character of $\bar{\phi^{\sigma}},$ and $\mathds{F}_p(\bar{\varphi^{\sigma}})$ is the field of character values. The way the process of reduction mod $p$ is defined implies that $\mathds{F}_p(\bar{\varphi^{\sigma}}) \subset \mathcal{O}_{\phi^{\sigma}}/ (\mathfrak{p} \cap \mathcal{O}_{\phi^{\sigma}}) \cong \mathcal{O}_{\phi}/ (\sigma^{-1}\mathfrak{p} \cap \mathcal{O}_{\phi}),$ where $\mathcal{O}_{\phi}$ is the ring of integers of the number field $E_{\phi}.$ For short denote the prime ideal $\sigma^{-1}\mathfrak{p} \cap \mathcal{O}_{\phi}$ by $\mathfrak{q}.$ Hence for each $\bar{\phi^{\sigma}} \in \overline{Orb(\phi)}$ we have 
\[|Orb(\bar{\phi^{\sigma}})| \leq |(\mathcal{O}_{\phi}/ \mathfrak{q}) : \mathds{F}_p|,\]
where $|Orb(\bar{\phi^{\sigma}})|$ means the same as before.
Therefore the number of orbits of the action of $\Gal( \bar{K} / \mathds{F}_p )$ on the set $\overline{Orb(\phi)}$ is at least

\[\frac{|\overline{Orb(\phi)}|}{|(\mathcal{O}_{\phi}/\mathfrak{q}) : \mathds{F}_p|}=\frac{|\Gal(E_{\phi} / \mathds{Q})|}{|\Gal( (\mathcal{O}_{\phi}/\mathfrak{q}) / \mathds{F}_p )|}=\frac{|\Gal(E_\phi / \mathds{Q})|}{|\mathfrak{D}(\mathfrak{q}/p)|} |\mathfrak{I}(\mathfrak{q}/p)|=t_{\rho} |\mathfrak{I}(\mathfrak{q}/p)| .\]
The last equality holds since $E_{\phi}=E_{\rho}$ and the Galois group $\Gal(E_\phi / \mathds{Q})$ acts transitively on prime ideals in $\mathcal{O}_{\phi}$ dividing $p$, and $\mathfrak{D}(\mathfrak{q}/p)$ is exactly the stabilizer of a prime ideal $\mathfrak{q}$ by $\Gal(E_\phi / \mathds{Q}).$

Since $\bar{K}$ is a splitting field for $G$ (\cite[(83.7)]{CuRe}) and $\bar{K}$ is a finite Galois extension of $\mathds{F}_p,$ we may apply Theorem~\ref{orbit} to $k=\mathds{F}_p$ and $E=\bar{K}.$ As before it gives one to one correspondence between irreducible $\mathds{F}_p$-representations of $G$ and orbits of the $\Gal(\bar{K} / \mathds{F}_p)$-action on the set of irreducible $\bar{K}$-representations of $G$. 

From what is computed above it follows that an irreducible rational representation $\rho \in I_p$ gives rise to at least $t_{\rho} |\mathfrak{I}(\mathfrak{q}/p)|$ different irreducible $\mathds{F}_p$-representations of $G,$ that correspond to orbits of $\Gal(\bar{K} / \mathds{F}_p)$-action on $\overline{Orb(\phi)},$ where $\phi$ is an irreducible $K$-representation occurring in decomposition of $K \otimes_{\mathds{Q}} \rho.$ Moreover, Proposition~\ref{Injectivity} and Theorem~\ref{orbit} guarantee that sets of corresponding $\mathds{F}_p$-representations coming from two non-isomorphic representations $\rho_1, \rho_2  \in I_p$ do not have common elements, because the corresponding sets $\overline{Orb(\phi_1)}$ and $\overline{Orb(\phi_2)}$ have empty intersection, where $\phi_1, \phi_2$ are irreducible $K$-representations occurring in decomposition of $K \otimes_{\mathds{Q}} \rho_1, K \otimes_{\mathds{Q}} \rho_2,$ respectively. Therefore we get  
\[\# \{ \text{irreducible $\mathds{F}_{p}$-representations of $G$} \} \geq \sum_{\rho \in I_p}  t_{\rho} |\mathfrak{I}(\mathfrak{q}/p)| \geq \sum_{\rho \in I_p}  t_{\rho}.\]

\end{proof}

We have the following corollary from Theorem~\ref{Inequality}.
\begin{corollary}\label{ranks-comparison}
For any finite group $G$ the following inequality holds
\[P(G) \geq R(G).\]
\end{corollary}
\begin{proof}
Note that
\[w_{\rho}=\sum_{p | \omega_{\rho}} \#\{\text{prime ideals in } \mathcal{O}_{\rho} \text{ that divide } p\}=\sum_{p | \omega_{\rho}} t_{\rho}\]
\[v_{\rho}=\sum_{p |~ |G|} \#\{\text{prime ideals in } \mathcal{O}_{\rho} \text{ that divide } p\}=\sum_{p |~ |G|} t_{\rho}\]

\[P(G)-R(G)=\varepsilon+\sum_{\rho \in X(G)} (w_{\rho}-v_{\rho})=\sum_{p \mid~ |G|}( \# \{ \text{irr. $\mathds{F}_{p}$-rep. of $G$} \}-\sum_{\substack{\rho \in X(G)\\ p  \nmid \omega_{\rho}}} t_{\rho})\]
\[=\sum_{p \mid~ |G|}( \# \{ \text{irreducible $\mathds{F}_{p}$-representations of $G$} \}-\sum_{\rho \in I_p} t_{\rho}) \geq 0,\]
where the last inequality holds by Theorem~\ref{Inequality} applied to each summand.
\end{proof}

\section{Concluding remarks}

In addition to what is done in Section 2 it is easy to compute both values $P(G), R(G)$ using the description provided in Section 1. For $G=\SL(2, \mathds{F}_3)$ these are $P(G)=6, R(G)=5.$

Using the computer algebra system GAP \cite{GAP4} we computed the difference $P(G)-R(G)$ for all finite groups of order less than 200. It turned out that all the groups that violate the condition $P(G)=R(G)$ (and are automatically counterexamples to the HTW-decomposition) have even order. Therefore it still seems reasonable to expect the Hambleton-Taylor-Williams Conjecture to be true for $G$ a finite group of odd order. By the famous Feit-Thompson Theorem \cite{FeTh} every finite group of odd order is solvable. Therefore expecting the HTW-decomposition for groups of odd order we do not go beyond the class of solvable groups.

All finite groups of order less than 24 satisfy the HTW-conjecture, so $G=\SL(2, \mathds{F}_3)$ is the smallest counterexample.  

If it is possible to correct the HTW-decomposition following the same pattern but choosing different numbers $\omega_{\rho}$ to be inverted, then these new numbers $\omega_{\rho}$ should satisfy the following relation
$$\sum_{\rho \in X(G)} \#\{\text{prime ideals in } \mathcal{O}_{\rho} \text{ that divide } \omega_{\rho}\}-\sum_{\rho \in X(G)} v_{\rho}+ \varepsilon=0.$$

The inequality $P(G) \geq R(G)$ proved in Corollary~\ref{ranks-comparison} leads to the natural guess that the weaker version of the HTW-Conjecture may hold. Namely, instead of asking for the isomorphism in the HTW-decomposition, one might conjecture that there is an injective homomorphism  $$G_n(\mathds{Z} G) \xhookrightarrow{} \bigoplus_{\rho \in X(G)} G_n(\Lambda_{\rho}), ~~\forall n \geq 0.$$

\end{document}